\documentclass[11pt]{article}

\usepackage{mathtools}
\usepackage{amsmath}
\usepackage{amsthm}
\usepackage{amssymb}
\usepackage{mathrsfs}
\usepackage{graphicx}
\usepackage{upgreek}
\def\div{{\,\rm div \,}}

\def\LL{{\,\rm L \,}}

\def\sign{{\,\rm sign \,}}

\def\T{{\cal T}}
\def\Id{{\,\rm Id \,}}

\def\CC{{\,\rm C\,}}
\def\WW{{\,\rm W\,}}

\def\qaq{{\quad\mbox{at}\quad}}

\def\id{{\,\rm id \,}}

\def\sym{{\,\rm sym \,}}

\def\ii{{\,\rm i \,}}

\def\+M{{\,\rm M^{n\times n}_+ \,}}
\def\tr{{\,\rm tr \,}}

\def\qfq{{\quad\mbox{for}\quad}}

\def\ii{{\,\rm i \,}}

\def\O{{\,\rm O\,}}

\def\lam{\lambda}

\def\L{{\cal L}}

\def\V{{\cal V}}

\def\O{{\cal O}}

\def\V{{\cal V}}

\newfont{\Blackboard}{msbm10 scaled 1200}

\newfont{\roma}{cmr10 scaled 1200}

\def\<{{\langle}}
\def\>{{\rangle}}

\def\Ga{\Gamma}

\def\si{\sigma}

\def\a{\alpha}
\def\b{\beta}

\newtheorem{thm}{{}\hskip\parindent Theorem}[section]
\newtheorem{lem}{{}\hskip\parindent Lemma}[section]
\newtheorem{pro}{{}\hskip\parindent Proposition}[section]

\newtheorem{cor}{{}\hskip\parindent Corollary}[section]

\def\dfrac{\displaystyle\frac}

\def\rw{\rightarrow}

\def\na{\nabla}

\def\be{\begin{equation}}
\def\ee{\end{equation}}
\def\beq{\arraycolsep=1.5pt\begin{eqnarray}}
\def\eeq{\end{eqnarray}}

\def\R{I\!\!R}

\def\n{\vec{n}}
\title{Strain Tensors and  Matching Property on Degenerated Hyperbolic Surfaces}
\date{}
\author{
Liang-Biao Chen and Peng-Fei Yao\\[0.2cm]
\nonumber
Key Laboratory of Systems and Control\\\nonumber
Institute of Systems Science,
Academy of Mathematics and Systems Science\\\nonumber
Chinese Academy of Sciences, Beijing 100190, P. R.
China\\\nonumber
School of Mathematical Sciences\\\nonumber
University of Chinese Academy of Sciences, Beijing 100049,
China\\\nonumber
e-mail: pfyao@iss.ac.cn}
\begin{document}
\maketitle
 \footnote{This work is supported by the National
Science Foundation of China, grants no. 12071463 and Key Research Program of Frontier Sciences,
CAS, no. QYZDJ-SSW-SYS011.}

\begin{quote}
\begin{small}
{\bf Abstract} \,\,\,We prove the regularity of solutions to the strain tensor equation on degenerated hyperbolic surfaces $S$ where  the Gauss curvature is zero on a part of boundary.
Furthermore, we obtain the density property that smooth infinitesimal isometries are dense in the $W^{2,2}({S},\R^3)$ infinitesimal isometries. Finally, the matching property is established. Those results are  important tools in
obtaining recovery sequences ($\Ga$-lim sup inequality) for dimensionally-reduced shell theories
in elasticity.
\\[3mm]
\end{small}
\end{quote}

\section{Introduction and Main Results}
\def\theequation{1.\arabic{equation}}
\hskip\parindent
Let $M\subset\R^3$ be a surface with a normal $\n$ and let the middle surface of a shell be  an open set $S\subset M.$ Let $T^kS$ denote  all the $k$-order tensor fields on $S$ for an integer $k\geq0.$ Let $T^2_{\sym}S$ be all the $2$-order symmetrical tensor fields on $S.$ For $y\in \WW^{1,2}(S,\R^3),$ we decompose it into $y=W+w\n,$ where $w=\<y,\n\>$ and $W\in TS.$
For $U\in T^2_\sym S$ given, linear strain tensor of a displacement $y\in\WW^{1,2}({S},\R^3)$ of the middle surface ${S}$ takes the form
\be \sym DW+w\Pi=U\qfq x\in{S},\label{01}\ee 
where $D$ is the connection of the induced metric in $M,$  $2\sym DW=DW+D^TW,$ and $\Pi$ is the second fundamental form of $M.$
Equation (\ref{01}) plays a fundamental role in the theory of thin shells, see \cite{HoLePa, LeMoPa,LePa, LeMoPa1,Yao2017, Yao2018} and many others. When $U=0,$ a solution $y$ to (\ref{01}) is referred to as an {\it infinitesimal isometry.}

The type of equation (\ref{01}) depends on the sign of the curvature on the region S: It is elliptic if $S$ has positive curvature; it is parabolic if the curvature is zero but $\Pi\not=0$ on $S;$ it is hyperbolic if $S$ has negative curvature.

Here  we establish the regularity of solutions to (\ref{01}) when $S$ is a non-characteristic region where its curvature is zero on a part of boundary, that will be specified below.
Then it
is proved that smooth infinitesimal isometries are dense in the $W^{2,2}({S},\R^3)$ infinitesimal isometries on the region $S.$ Finally, the matching property is derived that
smooth enough infinitesimal isometries can be matched with higher
order infinitesimal isometries. Those results are  important tools in
obtaining recovery sequences ($\Ga$-lim sup inequality) for dimensionally-reduced shell theories
in elasticity, when the elastic energy density scales like $h^\b,$ $\b\in(2, 4),$ that is, intermediate
regime between pure bending ($\b=2$) and the von-K\'arm\'an regime ($\b=4$). Such results have been obtained for elliptic surfaces \cite{LeMoPa1}, developable surfaces \cite{HoLePa}, and  hyperbolic surfaces \cite{Yao2017, Yao2018}. A survey on this topic is presented in \cite{LePa}.

In this paper we study the degenerated hyperbolic equation (\ref{01}), which is equivalent to a degenerated hyperbolic scalar equation of the form
\be\<D^2w,Q^*\Pi\>+\dfrac1\kappa\<Dw, X_0\>+\kappa(\tr_g\Pi)w=\kappa f+\dfrac1\kappa\<X_0,F\>+\<DF,Q^*\Pi\>,\label{Q1}\ee for $x\in S,$ $\kappa\not=0,$ where $w\in\LL^2(S)$ is the unknown, $f\in\LL^2(S)$ and $F\in\LL^2(S,TS)$ are given,
and $\kappa$ is the Gauss curvature. When $\kappa=0,$ there are two terms to degenerate in (\ref{Q1}): The coefficient of $\<Dw-F,X_0\>$ becomes
infinite; the second derivative of $w$ along the direction of the nonzero principal curvature is zero in $\<D^2w,Q^*\Pi\>.$
Those situations challenge the analysis of (\ref{Q1}).

Here we employ the Bochner technique and the tensor analysis to cope with the degenerates in (\ref{Q1}), where some priori estimates near the zero curvature curve in Section 2 play the key role.

We state our main results as follows.
Let $S\subset M$ be given by
$$ {S}=\{\,\a(t,s)\,|\,(t,s)\in[0,a)\times(0,b)\,\},\quad a>0, b>0,$$ where $\a:$ $[0,a)\times[0,b]\rw M$ is an imbedding map which is a family of regular curves with two parameters $t,$ $s$ such that
\be \Pi(\a_t(t,s),\a_t(t,s))>0\quad\mbox{for all}\quad (t,s)\in[0,a)\times[0,b],\label{Pi2.801}\ee
where $\a(\cdot,s)$ is a closed curve with  period $a$ for each $s\in[0,b]$ and
$$\{\,\a(t,0)\,|\,t\in[0,a)\,\}$$ is a closed curve or  just one point.

 {\bf Curvature assumptions}\,\,\, Let $\kappa$ be the Gaussian curvature function on $M.$ We assume that $S$ satisfies the following curvature conditions:
\be\kappa(x)<0\qfq x\in S\cup\Ga_{0};\label{kappa+}\ee
\be\kappa=0,\quad D\kappa(x)\not=0\qfq x\in\Ga_b,\label{kappa0}\ee
where
$$\Ga_{b}=\{\,\a(t,b)\,|\,t\in[0,a)\,\},\quad \Ga_{0}=\{\,\a(t,0)\,|\,t\in[0,a)\,\}.$$

Our main results are the following.

\begin{thm}\label{t1.1}Let ${S}$ be  of class $\CC^{2,1}.$ For $U\in\WW^{2,2}({S},T^2_{\sym}S),$ there exists a solution $y=W+w\n\in\WW^{1,2}({S},\R^3)$ to equation $(\ref{01})$ satisfying the bounds
\be \|W\|_{\WW^{2,1}({S},TS)}^2+\|w\|_{\WW^{1,2}(S)}^2\leq C\|U\|_{\WW^{2,2}({S},T^2_{\sym}S)}^2.\label{t0}\ee
If, in addition, ${S}\in\CC^{m+1,1},$ $U\in\WW^{m+1,2}({S},T^2_{\sym}S)$ for some $m\geq2,$ then
$$ \|W\|_{\WW^{m+1,2}({S},TS)}^2+\|w\|_{\WW^{m,2}({S})}^2\leq C\|U\|_{\WW^{m+1,2}({S},T^2_{\sym}S)}^2.$$
\end{thm}

By the imbedding theorem \cite[P. 158]{GNT}, the following corollary is immediate.

\begin{cor}\label{c1.1} Let $m\geq0$ be an integer and let $S$ be of ${S}\in\CC^{m+2,1}.$ Then problem $(\ref{01})$ admits a solution $y=W+w\n\in\CC^m_B(S,\R^3)$ satisfying
$$ \|W\|_{\CC^{m+1}_B({S},TS)}+\|w\|_{\CC^{m}_B(S)}\leq C\|U\|_{\CC^{m+3}_B({S},T^2_{\sym}S)},$$ where
$$\CC^m_B(S,\R^3)=\{\,y\in\CC^m(S,\R^3)\,|\,D^\a y\in\LL^\infty(S,\R^3)\,\,\mbox{for}\,\,|\a|\leq m\,\}.$$
\end{cor}

For $y\in\WW^{1,2}({S},\R^3),$ we denote the left hand side of equation (\ref{01}) by $\sym\nabla y.$
Let
$$\V({S},\R^3)=\{\,y\in\WW^{2,2}({S},\R^3)\,|\,\sym\nabla y=0\,\}.$$

\begin{thm}\label{t1.2} Let ${S}$ be  of class $\CC^{m+3,1}$  for some  integer $m\geq0.$ Then, for every $y\in\V({S},\R^3)$
there exists a sequence $\{\,y_k\,\}\subset\V({S},\R^3)\cap \CC^m_B({S},\R^3)$ such that
$$\lim_{k\rw\infty}\|y-y_k\|_{\WW^{2,2}({S},\R^3)}=0.$$
\end{thm}

A one parameter family $\{\,y_\varepsilon\,\}_{\varepsilon>0} \subset\CC^1_B(\overline{{S}},\R^3)$ is said to be a (generalized)
$m$th order infinitesimal isometry if the change of metric induced by $y_\varepsilon$ is of order $\varepsilon^{m+1},$ that is,
$$\|\nabla^Ty_\varepsilon\nabla y_\varepsilon-g\|_{L^\infty({S},T^2)}=\O(\varepsilon^{m+1})\quad\mbox{as}\quad \varepsilon\rw0,$$ where $g$ is the induced metric of $M$ from $\R^3,$ see \cite{HoLePa}.
A given $m$th order infinitesimal isometry can be modified by higher order
corrections to yield an infinitesimal isometry of order $m_1>m,$ a property to which we
refer to by {\it matching property of infinitesimal isometries}, \cite{HoLePa,LeMoPa1}. This property plays an important role in the construction of a recover sequence in the $\Ga$-limit for thin shells.

\begin{thm}\label{t1.3}Let ${S}$ be  of class $\CC^{4m,1}.$ Given $y\in\V({S},\R^3)\cap\CC^{4m-2}_B({S},\R^3),$
 there exists a family $\{\,z_\varepsilon\,\}_{\varepsilon>0}\subset\CC^2_B({S},\R^3),$ equi-bounded in $\CC^2_B({S},\R^3),$ such
that for all small $\varepsilon>0$  the family:
$$ y_\varepsilon=\id+\varepsilon y+\varepsilon^2z_\varepsilon$$
is a  $m$th order infinitesimal isometry of class $\CC^2_B({S},\R^3).$
\end{thm}

\setcounter{equation}{0}
\def\theequation{2.\arabic{equation}}
\section{Some priori estimates near the zero curvature curve}
\hskip\parindent Let $\nabla$ and $D$ denote the connection of $\R^3$ in the Euclidean metric and the one of $M$ in the reduced metric,
respectively. We have to treat the relationship between
$\nabla$ and $D$ carefully.

Let $m\geq1$ be an integer. Let $T\in T^mM$ be a $m$th order tensor field on $M.$ We define a $m-1$th order tensor field by
$$\ii_YT(Y_1,\cdots,Y_{m-1})=T(Y,Y_1,\cdots,Y_{m-1})\qfq Y_1,\,\,\cdots,\,Y_{m-1}\in TM,$$ which is called an {\it inner product} of $T$ with $Y.$ For any $T\in T^2S$ and $\a\in T_xM,$
$$\tr_g\ii_\a  DT$$ is a linear functional on $T_xM,$ where $\tr_g\ii_\a DT$ is the trace of the 2-order tensor field $\ii_\a DT$ in the induced metric $g.$ Thus there is a vector, denoted by $\div_gT,$ such that
$$\<\div_gT,\a\>=\tr_g\ii_\a DT\qfq \a\in T_xM,\,\,x\in M.$$ Clearly, the above formula defines a vector field $\div_gT\in TM.$

We need a linear operator $Q$ (\cite{Yao2017}, \cite{Yao2018}) as follows. For each point
$p\in M$, the Riesz representation theorem implies that there
exists an isomorphism $Q:$ $T_pM\rw T_pM$ such that
\be \label{2.1}
\<\a,Q\b\>=\det\left(\a,\b,\vec{n}(p)\right)\qfq\a,\,\b\in T_pM.\ee
Let $e_1,$ $e_2$ be an orthonormal basis of $T_pM$ with
positive orientation, that is,
$$\det\Big(e_1,e_2,\n(p)\Big)=1.$$
Then $Q$ can be expressed explicitly by
\be Q\a=\<\a,e_2\>e_1-\<\a,e_1\>e_2\quad\mbox{for
all}\quad\a\in T_pM.\label{1.3n}\ee
Clearly, $Q$ satisfies
$$ Q^T=-Q,\quad Q^2=-\Id.$$
The operator $Q$ plays an important role in our analysis.

In the present section, we consider problem
\be\left\{\begin{array}{l}\div_gQ\nabla\n V=f_1,\\
\div_gV=f_2,\end{array}\right.\label{V2.4}\ee where  $V\in TS$ and  $f_i$ are functions on $S.$

For given $\varepsilon>0,$ set
$$S_{b-\varepsilon}=\{\,\a(t,s)\,|\,(t,s)\in[0,a)\times(b-\varepsilon,b)\,\},\quad \Ga_{b-\varepsilon}=\{\,\a(t,b-\varepsilon)\,|\,t\in[0,a)\,\}.$$

The main results of this section are the following.

\begin{thm}\label{tt2.1}  Let $S_{b-\varepsilon}$ be of class $\CC^{3,1}.$ For given $\varepsilon>0$ small, there is $\si_\varepsilon>0$ such that
for any solution $V\in TS_{b-\varepsilon}$ to $(\ref{V2.4}),$ the
following estimates hold true.
\be\si_\varepsilon\|V\|_{L^2(S_{b-\varepsilon},TS_{b-\varepsilon})}^2\leq \|f_1\|_{L^2(S_{b-\varepsilon})}^2+\|f_2\|_{L^2(S_{b-\varepsilon})}+\|V\|_{\LL^2(\Ga_{b-\varepsilon},TM)}^2,\label{t2.71}  \ee
\beq\si_\varepsilon\|V\|_{\WW^{1,2}(S_{b-\varepsilon},TS_{b-\varepsilon})}^2&&\leq \|f_1\|_{\WW^{1,2}(S_{b-\varepsilon})}^2+\|f_2\|_{\WW^{1,2}(S_{b-\varepsilon})}^2+\|\ii_\nu DV\|_{\LL^2(\Ga_b)}^2\nonumber\\
&&\quad+\|V\|_{L^2(\Ga_\varepsilon\cup\Ga_{b},TM)}^2+\|DV\|_{L^2(\Ga_{-\varepsilon},T\Ga_{-\varepsilon})}^2, \label{t2.72}  \eeq where $\nu$ is the outside normal of $S_{b-\varepsilon}.$
\end{thm}

We make some preparations first. The proof of Theorem \ref{tt2.1} will be given in the end of this section.

\begin{lem}\label{l2.1} We have
\be QD_XY=D_X(QY)\qfq X,\,Y\in TM.\label{2.3}\ee
\end{lem}

\begin{proof} It follows from (\ref{2.1}) that
\beq
\<Z,QD_XY\>&&=\det(Z,\nabla_XY,\vec{n})\\\nonumber
&&=X\det(Z,Y,\vec{n})-\det(\nabla_XZ,Y,\vec{n})-\det(Z,Y,\nabla_X\vec{n})\\\nonumber
&&=X\<Z,QY\>-\<\nabla_XZ,QY\>=\<Z,\nabla_X(QY)\>\qfq X,\,Y,\,Z\in TM,\nonumber
\eeq
this gives (\ref{2.3}).
\end{proof}

\begin{lem} The following formulas are true.
\be \<X,Y\>Z=\<Z,Y\>X+\<Z,QX\>QY\qfq X,\,Y,\,Z\in TS.\label{n2.11n}\ee
\end{lem}

\begin{proof}Let $p\in S$ be given. If $Y=0,$ then (\ref{n2.11n}) holds. We assume that $|Y|=1.$ Then $QY,$ $Y$ forms an orthonormal basis of $T_pS.$ Thus
$$\<\<Z,Y\>X+\<Z,QX\>QY,Y\>=\<\<Z,Y\>X,Y\>=\<\<X,Y\>Z,Y\>,$$
\beq\<\<Z,Y\>X+\<Z,QX\>QY,QY\>&&=\<Z,Y\>\<X,QY\>+\<Z,QX\>\nonumber\\
&&=\<Z,Y\>\<X,QY\>+\<Z,Y\>\<Y,QX\>+\<Z,QY\>\<QY,QX\>\nonumber\\
&&=\<\<X,Y\>Z,QY\>.\nonumber\eeq
Thus (\ref{n2.11n}) follows.
\end{proof}

\begin{lem}\label{div} Let $P\in T^2S.$
Let $X$ and $Y$ be vector fields and $f$ be a function. Then
\be
\div_g(PX)=\<P,DX\>+\<\div_gP,X\>,\label{divP}\ee
$$\div_g(fP)=f\div_gP+P^TDf.$$
\end{lem}

\begin{proof}
Let $\{e_1,e_2\}$ be a normal frame field at $p.$
We have
\beq
\div_g(PX)
=&&\sum \<D_{e_i}PX,e_i\>
=\sum e_iP(X,e_i)
=\sum [(D_{e_i}P)(X,e_i)+\<P D_{e_i}X,e_i\>]\nonumber\\
=&&\sum\limits_i DP(X,e_i,e_i)
+\sum\limits_{i,j}\<Pe_j,e_i\>DX(e_j,e_i)\nonumber\\
=&&\<P,DX\>+\<\div_gP,X\>\quad\mbox{at}\quad p,\nonumber
\eeq and
\beq
\<\div_g(fP),X\>
=&&\sum D(fP)(X,e_i,e_i)
=\sum D_{e_i}(fP)(X,e_i)\nonumber\\
=&&\sum [e_i(f)P(X,e_i)+f(D_{e_i}P)(X,e_i)]
\nonumber\\
=&&\<X,P^TDf+f\div_gP\>\quad\mbox{at}\quad p.
\nonumber
\eeq
\end{proof}

\begin{lem}\label{P} Let $P\in T^2S$ and
let $p\in S$ be given. Then the following identities hold.
\be
\<Qv,w\>P-\<Pv,w\>Q
=Qv\otimes Pw-Pv\otimes Qw,\label{2.4}
\ee
\be\<Qv,w\>QP+\<Pv,w\>\id=Qv\otimes QP w+Pv\otimes w \qfq v,\,\,w\in T_pS.\label{2.57}\ee
\end{lem}
\begin{proof} Set
$$\Lambda(v,w,P)=\<Qv,w\>P
-\<Pv,w\>Q
-Qv\otimes Pw
+Pv\otimes Qw.$$
Let $\{e_1,e_2\}$ be an orthonormal basis of $T_pS$ with positive orientation. Since $\Lambda(v,w,P)$ is linear with respect to $v,$ $w,$ and $P,$ respectively, for (\ref{2.4}) it suffices
to prove $\Lambda(v,w,P)=0$ for $v=e_i$, $w=e_j$, and $P=e_k\otimes e_l.$ From (\ref{1.3n}), we obtain
$$\<Qe_i,e_j\>=-\<e_1\wedge e_2,e_i\otimes e_j\>=\<e_1\wedge e_2,e_j\otimes e_i\>\qfq 1\leq i,\,j\leq2.$$ Then we compute
\beq
&&\Lambda(e_i,e_j,e_k\otimes e_l)(e_m,e_n)=\<\Lambda(e_i,e_j,e_k\otimes e_l)e_m,e_n\>\nonumber\\
&&=\<Qe_i,e_j\>\delta_{km}\delta_{ln}
-\delta_{ki}\delta_{lj}\<Qe_m,e_n\>
-\<Qe_i,e_m\>\delta_{kj}\delta_{ln}
+\delta_{ki}\delta_{lm}\<Qe_j,e_n\>
\nonumber\\
&&=
-\<e_1\wedge e_2,\,\,(
\delta_{km}\delta_{ln}e_i\otimes e_j
-\delta_{ki}\delta_{lj}e_m\otimes e_n
-\delta_{kj}\delta_{ln}e_i\otimes e_m
+\delta_{ki}\delta_{lm}e_j\otimes e_n)\>
\nonumber\\
&&=
-\<e_1\wedge e_2,\,\,[
\delta_{ln}e_i\otimes (\delta_{km}e_j-\delta_{kj}e_m)
-\delta_{ki}e_n\otimes(\delta_{lm}e_j-\delta_{lj}e_m)
]\>.\label{2.167}\eeq
One observes (\ref{2.167}) to have
\beq
\Lambda(e_i,e_j,e_k\otimes e_l)(e_m,e_n)
=-\Lambda(e_i,e_m,e_k\otimes e_l)(e_j,e_n).
\nonumber
\eeq
It immediately follows that
$$\Lambda(e_i,e_j,e_k\otimes e_l)(e_j,e_n)=0\qfq j=m.$$

Let $i=j\not=m.$ From (\ref{2.167}),  we have
\beq
\Lambda(e_i,e_i,e_k\otimes e_l)(e_m,e_n)
&&=
-e_1\wedge e_2[
\delta_{ln}e_i\otimes (\delta_{km}e_i-\delta_{ki}e_m)
-\delta_{ki}e_n\otimes(\delta_{lm}e_i-\delta_{li}e_m)
]
\nonumber\\
&&=
\delta_{ki}e_1\wedge e_2[
\delta_{ln}e_i\otimes e_m
+e_n\otimes(\delta_{lm}e_i-\delta_{li}e_m)
].
\nonumber
\eeq
If $n=i$, then
\beq
\Lambda(e_i,e_i,e_k\otimes e_l)(e_m,e_i)
&&=
\delta_{ki}\delta_{lm}e_1\wedge e_2(e_i\otimes e_i)=0.
\nonumber
\eeq
If $n\ne i$, then $n=m$ and
\beq
\Lambda(e_i,e_i,e_k\otimes e_l)(e_m,e_m)
&&=
\delta_{ki}e_1\wedge e_2[
\delta_{lm}(e_i\otimes e_m
+e_m\otimes e_i)
-\delta_{li}e_m\otimes e_m
]=0.
\nonumber
\eeq

Let $i\not=j$ and $j\not=m.$ Then $i=m.$ It follows from (\ref{2.167}) that
\beq
&&\Lambda(e_i,e_j,e_k\otimes e_l)(e_m,e_n)\nonumber\\
&&=-\delta_{km}
\<e_1\wedge e_2,\,\,
\delta_{ln} e_m\otimes e_j
-e_n\otimes(\delta_{lm}e_j-\delta_{lj}e_m)
\>\nonumber\\
&&=\left\{\begin{array}{l}-\delta_{km}\delta_{lj}
\<e_1\wedge e_2,\,\,
 e_m\otimes e_j
+e_j\otimes e_m
\>=0\quad\mbox{if}\quad n=j,\\
-\delta_{km}\delta_{lm}
\<e_1\wedge e_2,\,\,e_m\otimes e_j
-e_m\otimes e_j
\>=0\quad\mbox{otherwise}\quad n=m.\end{array}\right.\nonumber\eeq

Using (\ref{2.4}), we have
\beq
\Lambda(Qv,Qw,-QP Q)=-\<Qv,w\>QP Q
-\<Pv,w\>Q
+v\otimes QP w
-QPv\otimes w
=0.
\nonumber
\eeq Thus (\ref{2.57}) follows from $\Lambda(Qv,Qw,-QP Q)Q=0.$
\end{proof}

We further assume that
\be\det(\a_t,\a_s,\n)>0\qfq\overline{S}.\label{as}\ee For otherwise, we replace $\a(t,s)$ with $\a(-t,s).$

Let $x\in\Ga_b.$ Since $\kappa(x)=0$ and $D\kappa(x)\not=0,$ from \cite[Lemma 2.6]{Yao2020}, there exist vector fields $X_1,$ $X_2$ in a neighborhood of $x$
satisfying $\nabla\n X_i=\lam_iX_i,$ where $\lam_i$ are the principal curvatures. Clearly we may extend the vector fields $X_i$ to the region $\overline{S}_{b-\varepsilon}$ when $\varepsilon>0$ is given small.
We assume that $X_i$ are vector fields such that
\be\nabla\n X_i=\lam_iX_i,\quad|X_i|=1,\quad\<X_1,X_2\>=0\qfq x\in\overline{S}_{b-\varepsilon},\label{X2.82}\ee where
$$\lam_1>0\qfq x\in \overline{S}_{b-\varepsilon},$$
$$\lam_2<0\qfq x\in S_{b-\varepsilon},\quad\lam_2=0\qfq x\in\Ga_b.$$

\begin{lem} For given $\varepsilon>0$ small
$$X_2(\lam_2)\not=0\qfq x\in S_{b-\varepsilon}.$$
\end{lem}

\begin{proof} It will suffice to prove
$$X_2(\lam_2)\not=0\qfq x\in\Ga_b.$$

First, we claim that
$$\<Q\a_t,X_2\>(x)\not=0\qfq x\in\Ga_b.$$ If not, then $\<Q\a_t,X_2\>(x)=0$ implies that $X_2=\eta \a_t$ with $\eta\not=0,$ and thus
$$\eta^2\Pi(\a_t,\a_t)=\<\nabla_{X_2}\n,X_2\>=\lam_2(x)=0,$$ which contradicts  the assumption $\Pi(\a_t,\a_t)\not=0.$

In addition, assumption (\ref{kappa0}) implies
\be\<D\kappa,\a_t\>=0,\quad D\kappa=\eta Q\a_t\qfq x\in\Ga_b,\label{Pii}\ee   for some $\eta\not=0.$
 Thus we obtain
\be X_2(\lam_2)=\frac1{\lam_1}X_2(\kappa)=\frac\eta{\lam_1}\<X_2,Q\a_t\>\not=0\qfq x\in\Ga_b.\label{X21}\ee
\end{proof}

We assume that
\be X_2(\lam_2)>0\qfq x\in\overline{S}_{b-\varepsilon}.\label{Xlam2}\ee For otherwise, we replace $X_2$ with $-X_2.$ Furthermore, we assume that $X_1,$ $X_2$ has positive orientation. For otherwise, we replace $X_1$ with $-X_1.$ Thus
\be QX_2=X_1,\quad QX_1=-X_2\qfq x\in\overline{S}_{b-\varepsilon}.\label{2.76}\ee

 Let $Y\in TS_{b-\varepsilon}$ be given. Define
\be
\L_{Y}V=e^{-s\lam_2}[(\div_gQ\nabla\n V+\<V,Y\>)X_1+\lam_2(\div_gV)X_2],\quad\label{2.73}\ee for $V\in TS_{b-\varepsilon}$ and $s>0.$
For given $V,$ $W\in TS_{b-\varepsilon}$, we have
\beq
\<W,\L_{Y}V\>
&&=e^{-s\lam_2}
[(\div_gQ\nabla\n V+\<V,Y\>)\<W,X_1\>+\lam_2(\div_gV)\<W,X_2\>]\nonumber\\
&&=
\div_g(e^{-s\lam_2}\<W,X_1\>Q\nabla\n V+e^{-s\lam_2}\lam_2\<W,X_2\>V)-Q\nabla\n V[e^{-s\lam_2}\<W,X_1\>]\nonumber\\
&&\quad
-V[e^{-s\lam_2}\lam_2\<W,X_2\>]
+e^{-s\lam_2}\<V,Y\>\<W,X_1\>\nonumber\\
&&=\<\L_{Y}^*W,V\>+\div_g(e^{-s\lam_2}\<W,X_1\>Q\nabla\n V+e^{-s\lam_2}\lam_2\<W,X_2\>V)\label{2.08+}
\eeq    for $x\in S_{b-\varepsilon},$
where
\beq
e^{s\lam_2}\<\L_{Y}^*W,V\>
&&=-\<DW,X_1\otimes Q\nabla\n V+\lam_2X_2\otimes V\>+s\<Q\nabla\n V,D\lam_2\>\<W,X_1\>\nonumber\\
&&\quad+s\lam_2\<V,D\lam_2\>\<W,X_2\>-\<W,D_{Q\nabla\n V}X_1\>-\<V,D\lam_2\>\<W,X_2\>\nonumber\\
&&\quad
-\lam_2\<W,D_VX_2\>+\<V,Y\>\<W,X_1\>\qfq x\in S_{b-\varepsilon},\label{2.79}
\eeq where the following formulas have been used
$$DW(Z_1,Z_2)=\<DW,Z_1\otimes Z_2\>\qfq Z_1,\,\,Z_2\in TS_{b-\varepsilon}.$$
On the other hand, from (\ref{2.73}) and Lemma \ref{div}, we obtain
\beq
e^{s\lam_2}\<V,\L_{Y}W\>&&=\<V,\,\,(\div_gQ\nabla\n W+\<W,Y\>)X_1+\lam_2(\div_gW)X_2\>\nonumber\\
&&
=\<DW,\<V,X_1\>Q\nabla\n+\lam_2\<V,X_2\>\id\>\nonumber\\
&&\quad
+\<\div_gQ\nabla\n+Y,W\>\<V,X_1\>\qfq x\in S_{b-\varepsilon}.\label{2.80}
\eeq

\begin{pro} For given $\varepsilon>0$ small, the following identities hold true.

$(i)$\,\,\,For any $x\in S_{b-\varepsilon},$
\be\<W,-\L_YV\>=\<-\L_{Y}^*W,V\>-\div_g(e^{-s\lam_2}\<W,X_1\>Q\nabla\n V+e^{-s\lam_2}\lam_2\<W,X_2\>V);\label{2.08}\ee

$(ii)$\,\,\,For any $X\in TS_{b-\varepsilon}$ with $|X|=1,$
\beq&&\<X_2,X\>\<(\<W,X_1\>Q\nabla\n V+\lam_2\<W,X_2\>V),\,\,X\>\nonumber\\
&&=\lam_2\<V,X\>\<W,X\>-\Pi(QX,QX)\<V,X_1\>\<W,X_1\>\qfq x\in S_{b-\varepsilon}.\label{Pi26}\eeq

\end{pro}

\begin{proof} (\ref{2.08}) follows from (\ref{2.08+}).

Next, we prove (\ref{Pi26}).
Using (\ref{n2.11n}) where $X=X_2,$ $Y=X,$ and $Z=V,$ we have
$$\<X_2,X\>V=\<V,X\>X_2+\<V,X_1\>QX\qfq x\in S_{b-\varepsilon}.$$ Thus we obtain
\beq&&\<X_2,X\>(\<W,X_1\>Q\nabla\n V+\lam_2\<W,X_2\>V)\nonumber\\
&&=\<W,X_1\>Q\nabla\n(\<V,X\>X_2+\<V,X_1\>QX)+\lam_2\<W,X_2\>(\<V,X\>X_2+\<V,X_1\>QX)\nonumber\\
&&=\lam_2\<W,X_1\>\<V,X\>X_1+\<W,X_1\>\<V,X_1\>Q\nabla\n QX+\lam_2\<W,X_2\>(\<V,X\>X_2+\<V,X_1\>QX)\nonumber\\
&&=\lam_2\<V,X\>W+\<W,X_1\>\<V,X_1\>Q\nabla\n QX+\lam_2\<W,X_2\>\<V,X_1\>QX\qfq x\in S_{b-\varepsilon}.\nonumber\eeq
It follows by $\<QX,X\>=0$ that
\beq&&\<X_2,X\>\<(\<W,X_1\>Q\nabla\n V+\lam_2\<W,X_2\>V),\,\,X\>\nonumber\\
&&=\lam_2\<V,X\>\<W,X\>-\Pi(QX,QX)\<V,X_1\>\<W,X_1\>\qfq x\in S_{b-\varepsilon}.\nonumber\eeq
\end{proof}

\begin{pro}\label{p2.2} Let $Y\in TS$ and let $\L_{Y}$ be given in $(\ref{2.73}).$ Then there exists a constant $\si_s>0$ such that
\be -e^{s\lam_2}\<\L_{Y}W+L^*_YW,W\>\geq\si_s|W|^2\qfq W\in TS_{b-\varepsilon},\quad x\in S_{b-\varepsilon}\label{2.86}\ee
 for $s>0$ large and $\varepsilon>0$ small enough.
\end{pro}

\begin{proof} Using Lemma \ref{div}, (\ref{2.73}), (\ref{2.76}), and (\ref{2.57}), we obtain
\beq &&e^{s\lam_2}\<\L_{Y}W,W\>=(\div_gQ\nabla\n W+\<W,Y\>)\<W,X_1\>+\lam_2(\div_gW)\<W,X_2\>\nonumber\\
&&=\<DW,\,\<W,X_1\>Q\nabla\n+\lam_2\<W,X_2\>\id\>+\<W,\div_g(Q\nabla\n)+Y\>\<W,X_1\>\nonumber\\
&&=\<DW,\,X_1\otimes Q\nabla\n W+\lam_2X_2\otimes W\>+\<W,\div_g(Q\nabla\n)+Y\>\<W,X_1\>.\label{2.83}\eeq
It follows from (\ref{2.79}) and (\ref{2.83}) that
\beq &&e^{s\lam_2}\<\L_{Y}W+\L_{Y}^*W,W\>=s\<Q\nabla\n W,D\lam_2\>\<W,X_1\>+s\lam_2\<W,D\lam_2\>\<W,X_2\>\nonumber\\
&&\quad-\<W,D_{Q\nabla\n W}X_1\>-\<W,D\lam_2\>\<W,X_2\>
-\lam_2\<W,D_WX_2\>+\<W,Y\>\<W,X_1\>\nonumber\\
&&\quad+\<W,\div_g(Q\nabla\n)+Y\>\<W,X_1\>\nonumber\\
&&=-sX_2(\lam_2)[\lam_1\<W,X_1\>^2-\lam_2\<W,X_2\>^2]-X_2(\lam_2)\<W,X_2\>^2\nonumber\\
&&\quad+2s\lam_2X_1(\lam_2)\<W,X_1\>\<W,X_2\>-X_1(\lam_2)\<W,X_1\>\<W,X_2\>\nonumber\\
&&\quad-\<W,D_{Q\nabla\n W}X_1\>
-\lam_2\<W,D_WX_2\>+\<W,Y\>\<W,X_1\>\nonumber\\
&&\quad+\<W,\div_g(Q\nabla\n)+Y\>\<W,X_1\>.\label{2.85}\eeq

Since $\lam_2=\O(\varepsilon)$ on $S_{b-\varepsilon},$ (\ref{2.86}) follows from (\ref{Xlam2}) and (\ref{2.85}), when $s>0$ is large enough and $\varepsilon>0$ is small enough, respectively.
\end{proof}

\begin{lem}For any $X,$ $Y\in TS,$ the following holds.
\be\nabla\vec{n}[X,Y]=D_X\nabla\vec{n}Y-D_Y\nabla\vec{n}X,\label{2.92}\ee
\be \div_g[X,Y]=X\div_gY-Y\div_gX,\label{div[X,Y]}\ee
\end{lem}

\begin{proof} A direct calculation yields
$$D_X\nabla\vec{n}Y-D_Y\nabla\vec{n}X=\nabla_X\nabla_Y\vec{n}-\nabla_Y\nabla_X\vec{n}-\<\nabla_X\nabla_Y\vec{n}-\nabla_Y\nabla_X\vec{n},\n\>=\nabla_{[X,Y]}\vec{n}.$$

For $v\in\CC^1_0(S)$, we have
\beq
-\int_Sv\div_g[X,Y]dg
&&=\int_S[X,Y]vdg
=\int_S(XYv-YXv)dg\nonumber\\
&&=\int_S[-(Yv)\div_gX+(Xv)\div_gY]dg
=\int_Sv(Y\div_gX-X\div_gY)dx.\nonumber
\eeq Then (\ref{div[X,Y]}) follows.
\end{proof}

Let $\Phi\in TS$ be given such that
$$\Pi(\Phi,\Phi)\ne0\qfq p\in \overline{S}_{b-\varepsilon}.$$  Define
\be RV=D_VQ\nabla\n\Phi-D_{Q\nabla\n V}\Phi\qfq V\in TS_{b-\varepsilon}.\label{2.94}\ee Set
\be h_1=\frac{\<R\Phi,Q\Phi\>}{\Pi(\Phi,\Phi)},\quad
h_2=\frac{\<RQ\Phi,Q\Phi\>
-h_1\<\nabla\n Q\Phi,\Phi\>}{|\Phi|^2},\label{h2.31}\ee
\be Z=\frac1{|\Phi|^2}(R-h_1Q\nabla\n-h_2\id)^T\Phi.\label{2.96}\ee

\begin{lem}The following formula is true.
\be R=h_1Q\nabla\n+h_2\id+Z\otimes\Phi.\label{2.97}\ee
\end{lem}

\begin{proof} We have
$$\<(R-h_1Q\nabla\n-h_2\id)\Phi,Q\Phi\>=h_1\Pi(\Phi,\Phi)-h_1\Pi(\Phi,\Phi)=0.$$
Since $\frac{Q\Phi}{|\Phi|},$ $\frac{\Phi}{|\Phi|}$ forms an orthonormal frame, it follows that
\beq
&&(R-h_1Q\nabla\n-h_2\id)W=\frac1{|\Phi|^2}\<(R-h_1Q\nabla\n-h_2\id)W,\Phi\>\Phi\nonumber\\
&&\quad+\frac1{|\Phi|^4}\<(R-h_1Q\nabla\n-h_2\id)(\<W,\Phi\>\Phi+\<W,Q\Phi\>Q\Phi,\,\,Q\Phi\>Q\Phi\nonumber\\
&&=\<W,Z\>\Phi\qfq W\in TS_{b-\varepsilon},\nonumber
\eeq  where $\<\Phi,Q\Phi\>=0.$
Thus (\ref{2.97}) follows.
\end{proof}

\begin{pro} Let $V\in TS_{b-\varepsilon}$ be a solution to problem $(\ref{V2.4}).$ Let $\Phi\in TS_{b-\varepsilon}$ be given. Then
\be \L_{Z}[\Phi,V]=e^{-s\lam_2}[(\Phi f_1-h_1f_1-h_2f_2+\<H,V\>)X_1
+\lam_2(\Phi f_2-\<D\div_g\Phi,V\>)X_2],\label{2.99}\ee where $Z$ is given in $(\ref{2.96})$ and
$$H=\nabla\n QD\div_g\Phi+h_1\div_gQ\nabla\n-\div_gR-\ii_ZD\Phi.$$
\end{pro}

\begin{proof} From (\ref{2.92}),  (\ref{2.94}), (\ref{div[X,Y]}), (\ref{divP}), and (\ref{2.97}), we have
\beq\div_gQ\nabla\n[\Phi,V]&&=\div_gQ(D_\Phi\nabla\n V-D_V\nabla\n\Phi)=\div_g(D_\Phi Q\nabla\n V-D_V Q\nabla\n\Phi)\nonumber\\
&&=\div_g[\Phi,Q\nabla\n V]-\div_gRV\nonumber\\
&&=\Phi\div_gQ\nabla\n V-Q\nabla\n V\div_g\Phi-\<h_1Q\nabla\n+h_2\id+Z\otimes\Phi,\,\,DV\>\nonumber\\
&&\quad-\<\div_gR,V\>\nonumber\\
&&=\Phi f_1-Q\nabla\n V\div_g\Phi-h_1(\<Q\nabla\n,DV\>+\<\div_gQ\nabla\n,V\>)-h_2\div_gV\nonumber\\
&&\quad-\<D_\Phi V-D_V\Phi,Z\>+\<h_1\div_gQ\nabla\n,V\>-D\Phi(Z,V)\nonumber\\
&&=-\<[\Phi,V],Z\>+\<\nabla\n QD\div_g\Phi+h_1\div_gQ\nabla\n-\ii_ZD\Phi-\div_gR,\,\,V\>\nonumber\\
&&\quad+\Phi f_1-h_2f_2,\nonumber\eeq and
$$\div_g[\Phi,V]=\Phi\div_gV-V\div_g\Phi=\Phi f_2-\<D\div_g\Phi,V\>.$$
Thus (\ref{2.99}) follows.
\end{proof}

\begin{lem}\label{l3.4}   $\<X_2,Q\a_t\><0$ for $x\in\Ga_b,$ where $X_2$ is given in $(\ref{X2.82}).$
\end{lem}

\begin{proof} Since $Q\a_t/|\a_t|,$ $\a_t/|\a_t|$ forms an orthonormal vector basis long $\Ga_0$ with positive orientation, the assumption (\ref{as}) and the curvature conditions (\ref{kappa+}) and (\ref{kappa0})
imply
\be D\kappa=\eta Q\a_t,\quad \eta\not=0,\quad \<D\kappa,-Q\a_t\>>0\qfq x\in\Ga_b.\label{2.35}\ee Then
$$\eta=\<D\kappa,Q\a_t\>/|\a_t|^2<0\qfq x\in\Ga_b.$$
It follows from (\ref{Xlam2}) that
$$\<X_2,Q\a_t\>=\frac1\eta\<X_2,D\kappa\>=\frac{\lam_1}\eta X_2(\lam_2)<0\qfq x\in\Ga_b.$$
\end{proof}

{\bf Proof of Theorem \ref{tt2.1}.}\,\,\,{\bf Step 1.}\,\,\, Let $V\in TS_{b-\varepsilon}$ satisfy (\ref{V2.4}). Taking $Y=0$ and $W=V$ in (\ref{2.73}),  (\ref{2.08}), and (\ref{Pi26}), respectively, we have, by (\ref{V2.4}),
$$\L_{0}V=e^{-s\lam_2}(f_1X_1+\lam_2f_2X_2),$$
\be\<V,-\L_0V\>=\<-\L_{0}^*V,V\>-\div_g(e^{-s\lam_2}\<V,X_1\>Q\nabla\n V+e^{-s\lam_2}\lam_2\<V,X_2\>V),\label{V2.38}\ee and
\beq&&\<X_2,X\>\<(\<V,X_1\>Q\nabla\n V+\lam_2\<V,X_2\>V),\,\,X\>\nonumber\\
&&=\lam_2\<V,X\>^2-\Pi(QX,QX)\<V,X_1\>^2\qfq X\in TS_{b-\varepsilon},\label{V2.39}\eeq
 respectively. Since $Q\a_t/|\a_t|,$ $\a_t/|\a_t|$ forms an orthonormal vector basis along $\Ga_{b-\varepsilon}$ and $\Ga_b,$ respectively, it follows from (\ref{as}) that
 $$\nu=-Q\a_t/|\a_t|\qfq x\in\Ga_b;\quad \nu=Q\a_t/|\a_t|\qfq x\in \Ga_{b-\varepsilon}.$$ When given $\varepsilon>0$ is small, from Lemma \ref{l3.4}, we have
 \be\<X_2,\nu\>>0\qfq x\in\Ga_b;\quad \<X_2,\nu\><0\qfq x\in\Ga_{b-\varepsilon}.\label{V40}\ee

We integrate (\ref{V2.38}) over $S_{b-\varepsilon}$ and use (\ref{V2.39}), where $X=\nu,$  Proposition \ref{p2.2}, and (\ref{V40}) to obtain
\beq
&&2(V,-\L_0V)_{\LL^2(S_{b-\varepsilon},TS_{b-\varepsilon})}
=(V,-\L_0^*V-\L_0V)_{\LL^2(S_{b-\varepsilon},TS_{b-\varepsilon})}\nonumber\\
&&\quad+\int_{\Ga_{b-\varepsilon}\cup\Ga_{b}}
\frac{e^{-\lam_2s}}{\<X_2,\nu\>}[\Pi(Q\nu,Q\nu)\<V,X_1\>^2-\lam_2\<V,\nu\>^2]d\Ga\nonumber\\
&&\geq\si_s\|V\|^2_{\LL^2(S_{b-\varepsilon},TS_{b-\varepsilon})}+\int_{\Ga_{b}}
\frac{e^{-\lam_2s}}{\<X_2,\nu\>}\Pi(Q\nu,Q\nu)\<V,X_1\>^2d\Ga\nonumber\\
&&\quad+\int_{\Ga_{b-\varepsilon}}
\frac{e^{-\lam_2s}}{\<X_2,\nu\>}[\Pi(Q\nu,Q\nu)\<V,X_1\>^2d\Ga-\lam_2\<V,\nu\>^2]d\Ga.\quad\label{2.42}\eeq
Thus (\ref{t2.71}) follows.

{\bf Step 2.}\,\,\, We prove (\ref{t2.72}).
Let
\be\Phi_1=\a_t,\quad \Phi_2=\Phi_1+\eta Q\Phi_1,\label{e2.37}\ee where $\eta>0$ is given small such that
\be\Pi(\Phi_i,\Phi_i)>0\qfq x\in S_{b-\varepsilon},\quad i=1,\,\,2.\label{e2.38}\ee Clearly, $\Phi_1,$ $\Phi_2$ forms a basis of vector fields. It  follows from (\ref{2.99}) that
\be \L_{Z_i}[\Phi_i,V]=e^{-s\lam_2}[(\Phi_i f_1-h_{i1}f_1-h_{i2}f_2+\<H_i,V\>)X_1
+\lam_2(\Phi_i f_2-\<D\div_g\Phi_i,V\>)X_2],\label{L2.44}\ee  where $Z_i$ is given in $(\ref{2.96})$ with $\Phi=\Phi_i,$ and
$$H_i=\nabla\n QD\div_g\Phi_i+h_{i1}\div_gQ\nabla\n-\div_gR_i-\ii_{Z_i}D\Phi_i.$$

We repeat the produce in Step 1 with $\L_0V$ replaced by $\L_{Z_i}[\Phi_i,V]$ to obtain (\ref{t2.72}).
\hfill$\Box$

\setcounter{equation}{0}
\section{The degenerated hyperbolic regions}
\def\theequation{3.\arabic{equation}}
\hskip\parindent
In the sequel we assume that $\Ga_{0}=\{\,\a(t,0)\,|\,t\in[0,a)\,\}$ is a non-degenerated curve, that does not contain just one point. For otherwise, we may replace $\Ga_{0}$ with any curve
$$\Ga_s=\{\,\a(t,s)\,|\,t\in[0,a)\,\}\qfq s\in(0,b).$$

Let
$$X_0=Q\nabla\n QD\kappa\qfq x\in S.$$
Let $F\in\WW^{1,2}(S,TS)$ and $f\in L^2(S).$ We consider a degenerated hyperbolic problem
\be\left\{\begin{array}{l}\<D^2w,Q^*\Pi\>+\dfrac1\kappa X_0w+\kappa(\tr_g\Pi)w=\kappa f+\dfrac1\kappa\<X_0,F\>+\<DF,Q^*\Pi\>\quad x\in S,\\
 w=q_0,\quad\<Dw,Q\nabla\n\a_t\>=q_1,\quad x\in\Ga_{0}.\end{array}\right.\label{4.1}\ee

\begin{thm}\label{t4.1}Let $m\geq0$ be an integer. Let $S$ be of class $\CC^{m,1}.$  Let $F\in\WW^{m+1,2}(S,TS),$ $f\in \WW^{m,2}(S),$ $q_0\in\WW^{m+1,2}(\Ga_0),$ and $q_1\in\WW^{m,2}(\Ga_0).$ Then problem $(\ref{4.1})$ admits a unique solution $w\in\WW^{m+1,2}(S)$ satisfying
\beq\|w\|^2_{\WW^{m+1,2}(S)}+&&\|w\|^2_{\WW^{m+1,2}(\Ga_b)}\leq C(\|f\|^2_{\WW^{m,2}(S)}+\|F\|^2_{\WW^{m+1,2}(S,TS)}\nonumber\\
&&\quad+\|q_0\|^2_{\WW^{m+1,2}(\Ga_0)}+\|q_1\|^2_{\WW^{m,2}(\Ga_0)}),\label{4.4}\eeq
\be \<Dw,X_0\>=\<F,X_0\>\qfq x\in\Ga_0.\label{X4.3}\ee
\end{thm}

Let
\be\Ga_2(w,S)=\int_0^a(|D^2w|^2+|Dw|^2+|w|^2)\circ\a(t,0)dt.\label{n3.7}\ee

\begin{thm}\label{t4.2}Let $S$ be of class $\CC^{2,1}.$ Let $w$ solve problem $(\ref{4.1}).$ Then there are constants $C>c>0$ such that
\beq c\Ga_2(w,S)&&\leq\|w\|^2_{\WW^{2,2}(S)}+\|f\|^2_{\WW^{1,2}(S)}+\|F\|^2_{\WW^{2,2}(S,TS)}\nonumber\\
&&\leq C(\Ga_2(w,S)+\|f\|^2_{\WW^{1,2}(S)}+\|F\|^2_{\WW^{2,2}(S,TS)}).\label{Ga4.6}\eeq
\end{thm}
The proofs of Theorems \ref{t4.1} and \ref{t4.2} will be given in the end of this section.

Fix $\varepsilon_0>0$ small, such that (\ref{Xlam2}), Proposition \ref{p2.2}, and
\be\<X_2,Q\a_t\><0\qfq x\in\overline{S}_{b-\varepsilon_0}\quad\mbox{(by Lemma \ref{l3.4})},\label{X24.2}\ee hold true.
The curve $\Ga_{b-\varepsilon_0}$ divides $S$ into two regions:
$$S=\Sigma_1\cup\Sigma_2\cup\Ga_{b-\varepsilon_0},$$ where
$$\Sigma_1=\{\,\a(t,s)\,|\,(t,s)\in[0,a)\times(0,b-\varepsilon_0)\,\},\quad\Sigma_2=\{\,\a(t,s)\,|\,(t,s)\in[0,a)\times(b-\varepsilon_0,b)\,\}.$$
We shall obtain solutions to the boundary-value problems on the regions $\Sigma_1$ and $\Sigma_2,$ separately. Then paste them together to  have solutions to  problem (\ref{4.1}) on
the region $S.$

To  apply some existence results in \cite{Yao2017,Yao2018} to the boundary-value problem   on the regions $\Sigma_i,$ we recall some boundary operators in  \cite{Yao2017}. Let $x\in\Ga_{b-\varepsilon_0}$ be given. $\mu\in T_xS$ with $|\mu|=1$ is said to be the {\it noncharacteristic normal} outside $\Sigma_2$ if there is a curve
$\zeta:$ $(0,\varepsilon)\rw \Sigma_2$ such that
$$\zeta(0)=x,\quad \zeta'(0)=-\mu,\quad\Pi(\mu,\b)=0\qfq \b\in T_x\Ga_{b-\varepsilon_0}.$$ Let $\mu$ be the the noncharacteristic normal field along $\Ga_0$ outside $\Sigma_2.$
 We define the boundary operators $\T_i:$ $T_xM\rw T_xM$ by
$$ \T_i\b=\frac{1}{2}\Big[\b+(-1)^i\chi(\mu,\b)\rho(\b)Q\nabla\n \b]\qfq \b\in T_xM,\quad i=1,\,\,2,$$ where
$$\chi(\mu,\b)=\sign\det\Big(\mu,\b,\n\Big),\quad \varrho(\b)=\frac{1}{\sqrt{-\kappa}}\sign\Pi(\b,\b),$$ and  $\sign$ is the sign function. Noting that
$$\mu=\frac{Q\nabla\n\a_t}{|\nabla\n\a_t|}\qfq x\in\Ga_{b-\varepsilon_0},$$ it follows from (\ref{Pi2.801}) and (\ref{as})  that
$$\chi(\mu,\a_t)\circ\a(t,b-\varepsilon_0)=1,\quad \varrho(\a_t)\circ\a(t,b-\varepsilon_0)=\frac{1}{\sqrt{-\kappa}}\qfq t\in[0,a).$$ Thus
$$(\T_2-\T_1)\a_t\circ\a(t,0)=-\frac1{\sqrt{-\kappa}}Q\nabla\n\a_t\qfq t\in[0,a).$$
Similarly, as $\Ga_{0}$ is  a part of the boundary of the region $\Sigma_1,$ then
$$(\T_2-\T_1)\a_t\circ\a(t,0)=\frac1{\sqrt{-\kappa}}Q\nabla\n\a_t\qfq t\in[0,a).$$

Since $\Sigma_1$ is non-degenerated hyperbolic region, by similar arguments as for \cite[Theorems 4.2 and 4.3]{Yao2017} (or \cite{Yao2018}), we have the following. The details are omitted.

\begin{pro}\label{p4.1} Problem
\be\left\{\begin{array}{l}\<D^2w,Q^*\Pi\>+\dfrac1\kappa X_0w+\kappa(\tr_g\Pi)w=\kappa f+\dfrac1\kappa\<X_0,F\>+\<DF,Q^*\Pi\>,\,\,x\in \Sigma_1,\\
w=q_0,\quad\<Dw,Q\nabla\n\a_t\>=q_1\qfq x\in\Ga_0\end{array}\right.\label{4.7}\ee admits a unique solution $w$ satisfying
\beq\|w\|^2_{\WW^{m+1,2}(\Sigma_1)}+&&\sum_{i=0}^{m+1}\|D^iw\|^2_{\LL^2(\Ga_{b-\varepsilon_0})}\leq C(\|f\|^2_{\WW^{m,2}(\Sigma_1)}+\|F\|^2_{\WW^{m+1,2}(\Sigma_1,T\Sigma_1)}\nonumber\\
&&\quad+\|q_0\|^2_{\WW^{m+1,2}(\Ga_0)}+\|q_1\|^2_{\WW^{m,2}(\Ga_0)}).\nonumber\eeq
\end{pro}

Consider a degenerated problem
\be\left\{\begin{array}{l}\<D^2w,Q^*\Pi\>+\dfrac1\kappa X_0w+\kappa(\tr_g\Pi)w=\kappa f+\dfrac1\kappa\<X_0,F\>+\<DF,Q^*\Pi\>,\,\,x\in \Sigma_2,\\
w=q_0,\quad\<Dw,Q\nabla\n\a_t\>=q_1\qfq x\in\Ga_{b-\varepsilon_0}.\end{array}\right.\label{4.8}\ee

Consier the regions
\be\Sigma_{2\varepsilon}=\{\,\a(t,s)\,|\,(t,s)\in[0,a)\times(b-\varepsilon_0,b-\varepsilon)\,\}\qfq 0<\varepsilon<\varepsilon_0.\label{sigma3.9}\ee Since $\Sigma_{2\varepsilon}$ are also non-degenerated hyperbolic regions for all given $0<\varepsilon<\varepsilon_0,$ for the same reasons as for $\Sigma_1,$ problem (\ref{4.8}) has solutions in $\WW^{m+1,2}(\Sigma_{2\varepsilon})$ on the regions
$\Sigma_{2\varepsilon}.$ Because $0<\varepsilon<\varepsilon_0$ can be arbitrarily small, problem (\ref{4.8}) actually admits a unique solution $w$ on the region $\Sigma_2,$ which satisfies
$w\in\WW^{m+1,2}(\Sigma_{2\varepsilon})$ for all $0<\varepsilon<\varepsilon_0.$ Then Theorem \ref{t4.1} follows from Propositions \ref{p4.1} and  \ref{p4.2} later.

\begin{lem} \label{X0} We have
$$-\kappa(\nabla\n)^{-1}\b=Q\nabla\n Q\b\qfq \b\in T_xM,\quad x\in M,\quad\kappa(x)\not=0.$$
\end{lem}

\begin{proof} Let $\kappa(x)\not=0.$ Let $e_1,$ $e_2$ be an orthonormal basis of $T_xS$ with positive orientation such that
$$\nabla\n e_i=\lam_ie_i\qfq x,$$ where $\lam_i$ are the principal curvatures. We have
\beq Q\nabla\n Q\b&&=Q\nabla\n Q(\<\b,e_1\>e_1+\<\b,e_2\>e_2)=Q\nabla\n(-\<\b,e_1\>e_2+\<\b,e_2\>e_1)\nonumber\\
&&=Q(-\<\b,e_1\>\lam_2e_2+\<\b,e_2\>\lam_1e_1)=-\<\b,e_1\>\lam_2e_1-\<\b,e_2\>\lam_1e_2\nonumber\\
&&=-\kappa(\nabla\n)^{-1}\b.\nonumber\eeq
\end{proof}

\begin{lem}\label{ll2}  For $V\in \WW^{1,2}(S,TS),$ we have
$$\<D(\nabla\n V),Q^*\Pi\>=\div_g\kappa V\qfq x\in S.$$
\end{lem}

\begin{proof} Let $x\in S$ be fixed. Let $e_1,$ $e_2$ be an orthonormal basis of $T_xS$ with positive orientation such that
$$\nabla\n e_1=\lam_ie_i.$$ Suppose that $E_1,$ $E_2$ be a frame field normal at $x$ such that
$$E_i=e_i\qaq x.$$ We have
\beq\<D(\nabla\n V),Q^*\Pi\>&&=\lam_2E_1\<\nabla\n V,E_1\>+\lam_1E_2\<\nabla\n V,E_2\>\nonumber\\
&&=\lam_2E_1\< V,\nabla\n E_1\>+\lam_1E_2\<V,\nabla\n E_2\>\nonumber\\
&&=\lam_2\<D_{E_1} V,\nabla\n E_1\>+\lam_1\<D_{E_2}V,\nabla\n E_2\>+\lam_2\< V,D_{E_1}(\nabla\n E_1)\>+\lam_1\<V,D_{E_2}(\nabla\n E_2)\>\nonumber\\
&&=\kappa\div_gV+(\Pi_{22}\Pi_{111}+\Pi_{11}\Pi_{122})\< V, E_1\>+(\Pi_{22}\Pi_{121}+\Pi_{11}\Pi_{222})\<V,E_2\>\nonumber\\
&&=\kappa\div_gV+(\Pi_{22}\Pi_{111}+\Pi_{11}\Pi_{221}-2\Pi_{12}\Pi_{121})\< V, E_1\>\nonumber\\
&&\quad+(\Pi_{22}\Pi_{112}+\Pi_{11}\Pi_{222}-2\Pi_{12}\Pi_{122})\<V,E_2\>\nonumber\\
&&=\kappa\div_gV+E_1(\kappa)\<V,E_1\>+E_2(\kappa)\<V,E_2\>=\div_g\kappa V\qaq x,\nonumber\eeq where $\Pi_{ij}=\Pi(E_i,E_j)$ and $\Pi_{ijk}=D\Pi(E_i,E_j,E_k).$
\end{proof}

\begin{lem} For $v,$ $w\in\WW^{2,2}(S),$ we have
\be\<D^2v,Q^*\Pi\>=\div_g\ii_{Dv}Q^*\Pi,\label{Pi6}\ee
\be \Pi(QDv,QDw)+w\<D^2v,Q^*\Pi\>=\div_g(w\ii_{Dv}Q^*\Pi).\label{Pi7}\ee
\end{lem}

\begin{proof} Let $x\in S$ be fixed. Suppose that $E_1,$ $E_2$ is a frame  field normal at $x$ with positive orientation such that
$$\Pi(E_1,E_2)=0\qaq x.$$ Then
$$QE_1=-E_2,\quad QE_2=E_1\quad\mbox{in a neighborhood of $x.$}$$
Using the above formulas, we compute
\beq\div_g\ii_{Dv}Q^*\Pi&&=E_1\<\ii_{Dv}Q^*\Pi,E_1\>+E_2\<\ii_{Dv}Q^*\Pi,E_2\>=E_1[\Pi(QDv,QE_1)]+E_2[\Pi(QDv,QE_2)]\nonumber\\
&&=E_1(v_1\Pi_{22}-v_2\Pi_{12})+E_2(-v_1\Pi_{12}+v_2\Pi_{11})\nonumber\\
&&=v_{11}\Pi_{22}+v_1\Pi_{221}-v_2\Pi_{121}-v_1\Pi_{122}+v_{22}\Pi_{11}+v_2\Pi_{112}\nonumber\\
&&=\<D^2v,Q^*\Pi\>,\nonumber\eeq where $v_i=E_iv,$ $\Pi_{ijk}=D\Pi(E_i,E_j,E_k),$ and the following formulas have been used
$$\Pi_{221}=\Pi_{122},\quad \Pi_{121}=\Pi_{112}.$$

Since $\<\ii_{Dv}Q^*\Pi,Dw\>=\Pi(QDv,QDw),$ (\ref{Pi7}) follows from (\ref{Pi6}).
\end{proof}

\begin{pro}\label{p4.2} Problem $(\ref{4.8})$ admits a unique solution $w$ satisfying
\beq\|w\|^2_{\WW^{1,2}(\Sigma_2)}+&&\|w\|^2_{\WW^{1,2}(\Ga_{b})}\leq C(\|f\|^2_{\LL^2(\Sigma_2)}+\|F\|^2_{\WW^{1,2}(\Sigma_2,T\Sigma_2)}\nonumber\\
&&\quad+\|q_0\|^2_{\WW^{1,2}(\Ga_{b-\varepsilon_0})}+\|q_1\|^2_{\LL^2(\Ga_{b-\varepsilon_0})}),\label{4.9}\eeq
\be \<Dw,X_0\>=\<F,X_0\>\qfq x\in\Ga_0.\label{4.10}\ee
Furthermore, for $m\geq1,$ the following estimates hold.
\beq\|w\|^2_{\WW^{m+1,2}(\Sigma_2)}+&&\|w\|^2_{\WW^{m+1,2}(\Ga_b)}\leq C(\|f\|^2_{\WW^{m,2}(\Sigma_2)}+\|F\|^2_{\WW^{m+1,2}(\Sigma_2,T\Sigma_2)}\nonumber\\
&&\quad+\|q_0\|^2_{\WW^{m+1,2}(\Ga_{b-\varepsilon_0})}+\|q_2\|^2_{\WW^{m,2}(\Ga_{b-\varepsilon_0})}).\label{w4.9}\eeq
\end{pro}

\begin{proof} Our task is to establish (\ref{4.9})-(\ref{w4.9}).
 Let $w$ solve problem (\ref{4.8}). As before, we define
\be W=(\nabla\n)^{-1}(Dw-F)\qfq x\in\Sigma_2.\label{M4.12}\ee Then
$$Dw=\nabla\n W+F\qfq x\in \Sigma_2.$$

Let $E_1,$ $E_2$ be an orthonormal frame on $S_{b-\varepsilon_0}$ with positive orientation. From Lemma \ref{l2.1}, we have
\beq\div_gQDw&&=\<D_{E_1}(QDw),E_1\>+\<D_{E_2}(QDw),E_2\>=\<QD_{E_1}Dw,E_1\>+\<QD_{E_2}Dw,E_2\>\nonumber\\
&&=-\<D_{E_1}Dw,QE_1\>-\<D_{E_2}Dw,QE_2\>=\<D_{E_1}Dw,E_2\>-\<QD_{E_2}Dw,E_1\>=0.\nonumber\eeq
Thus we obtain
$$\div_gQ\nabla\n W=-\div_gQF\qfq x\in \Sigma_2.$$

On the other hand, from Lemma \ref{X0}, we have
$$X_0=-\kappa(\nabla\n)^{-1}D\kappa\qfq x\in \Sigma_2.$$ Then
$$\<W,D\kappa\>=\<Dw-F,(\nabla\n)^{-1}D\kappa\>=\frac1\kappa\<F-Dw,X_0\>\qfq x\in \Sigma_2.$$
In addition, it follows from Lemma \ref{ll2} and the first equation in (\ref{4.8}) that
\beq\kappa\div_gW+\<W,D\kappa\>&&=\div_g\kappa W=\<D(\nabla\n W),Q^*\Pi\>=\<D^2w-DF,Q^*\Pi\>\nonumber\\
&&=\kappa[f-(\tr_g\Pi)w]+\frac1\kappa\<X_0,F-Dw\>\qfq x\in \Sigma_2.\nonumber\eeq Thus we obtain
$$\div_gW=f-(\tr_g\Pi)w\qfq x\in \Sigma_2.$$
That is, $W\in T\Sigma_2$ solves problem
\be\left\{\begin{array}{l}\div_gQ\nabla\n W=-\div_gQF\qfq x\in \Sigma_2,\\
\div_gW=-(\tr_g\Pi)w+f\qfq x\in \Sigma_2.\end{array}\right.\label{4.11}\ee

{\bf Step 1.}\,\,\, Consider the case of $m=0.$

We first prove (\ref{4.10}). Let $\L_0W$ be given by (\ref{2.73}) where $V=W$ and $Y=0.$ It follows from (\ref{4.11}) that
\be\L_0W=e^{-s\lam_2}\{-(\div_gQF)X_2+\lam_2[-(\tr_g\Pi)w+f]X_2\}.\label{L3.17}\ee From (\ref{2.08}) and (\ref{2.86}), we have
\be2\<W,-\L_0W\>\geq\si_{\varepsilon_0}|W|^2-\div_g(e^{-s\lam_2}\<W,X_2\>Q\nabla\n W+e^{-s\lam_2}\lam_2\<W,X_2\>W),\label{4.12}\ee for $x\in\Sigma_2.$
Let $\Sigma_{2\varepsilon}$ be given in (\ref{sigma3.9}) for $0<\varepsilon<\varepsilon_0.$ We integrate (\ref{4.12}) over $\Sigma_{2\varepsilon}$ and use (\ref{Pi26})
to have
\beq2(W,&&-\L_0W)_{\LL^2(\Sigma_{2\varepsilon},T\Sigma_{2\varepsilon})}\geq\si_{\varepsilon_0}\|W\|^2_{\LL^2(\Sigma_{2\varepsilon},T\Sigma_{2\varepsilon})}\nonumber\\
&&\quad
+\int_{\Ga_{b-\varepsilon}}\frac{e^{-s\lam_2}}{\<X_2,\nu\>}[\frac{\Pi(\a_t,\a_t)}{|\a_t|^2}\<W,X_1\>^2+|\lam_2|\<W,\nu\>^2]d\Ga\nonumber\\
&&\quad
+\int_{\Ga_{b-\varepsilon_0}}\frac{e^{-s\lam_2}}{\<X_2,\nu\>}[\frac{\Pi(\a_t,\a_t)}{|\a_t|^2}\<W,X_1\>^2+|\lam_2|\<W,\nu\>^2]d\Ga,\label{X4.13}\eeq
for $0<\varepsilon<\varepsilon_0.$

Next, we deal with the first boundary integration in the right hand side of (\ref{X4.13}). Since
$$\<W,\nu\>=\<\nu,X_2\>\<W,X_1\>+\frac{\<\nu,X_2\>}{\lam_2}\<Dw-F,X_2\>\qfq x\in\Sigma_{2\varepsilon},$$ it follows that
\beq&&\frac{\Pi(\a_t,\a_t)}{|\a_t|^2}\<W,X_1\>^2+|\lam_2|\<W,\nu\>^2\nonumber\\
&&\geq\frac{\<\nu,X_2\>^2}{2|\lam_2|}\<Dw-F,X_2\>^2+[\frac{\Pi(\a_t,\a_t)}{|\a_t|^2}-|\lam_2|\<\nu,X_1\>^2]\<W,X_1\>^2\qfq x\in\Ga_{b-\varepsilon}.\nonumber\eeq
Noting that
$$|\lam_2|=\O(\varepsilon),\quad\<X_2,\nu\>>0\qfq x\in\Ga_{b-\varepsilon},$$ we have
\beq&&\int_{\Ga_{b-\varepsilon}}\frac{e^{-s\lam_2}}{\<X_2,\nu\>}[\frac{\Pi(\a_t,\a_t)}{|\a_t|^2}\<W,X_1\>^2+|\lam_2|\<W,\nu\>^2]d\Ga\nonumber\\
&&\geq\si\int_{\Ga_{b-\varepsilon}}(\frac1{|\lam_2|}\<Dw-F,X_2\>^2+\<W,X_1\>^2)d\Ga\qfq \varepsilon>0\quad\mbox{small},\label{L3.20}\eeq which imply that
\be\<Dw, X_2\>=\<F,X_2\>\qfq x\in\Ga_b.\label{Ga3.21}\ee On the other hand, from (\ref{2.35}),
\beq X_0&&=Q\nabla\n QD\kappa=-\eta Q\nabla\n\a_t=\eta\<\nabla\n\a_t,QX_1\>X_1+\eta\<\nabla\n\a_t,QX_2\>X_2\nonumber\\
&&=-\eta\<\nabla\n\a_t,X_2\>X_1+\eta\<\nabla\n\a_t,X_1\>X_2=\eta\lam_1\<\a_t,X_1\>X_2\qfq x\in\Ga_b.\label{X03.21}\eeq
Thus (\ref{4.10}) follows.

 We claim that
\be\<X_1,\tau\>\not=0\qfq x\in\Ga_b,\label{tau3.22}\ee where $\tau=\a_t/|\a_t|.$ For otherwise, $\<X_1,\tau\>=0$ implies that $\tau=\pm X_2$ and
$$\Pi(\tau,\tau)=\Pi(X_2,X_2)=0,$$ which contradicts (\ref{Pi2.801}). Thus
\beq\<Dw,\tau\>&&=\<\tau, X_1\>\<Dw,X_1\>+\<\tau,X_2\>\<Dw,X_2\>\nonumber\\
&&=\<\tau, X_1\>\<Dw,X_1\>+\<\tau,X_2\>\<F,X_2\>,\nonumber\eeq that is
$$\<Dw,X_1\>=\frac1{\<X_1,\tau\>}(\<Dw,\tau\>-\<X_2,\tau\>\<F,X_2\>)\qfq x\in\Ga_b.$$ Then
\beq\<W,X_1\>&&=\frac1{\lam_1}\<Dw-F,X_1\>=\frac1{\lam_1\<X_1,\tau\>}\<Dw,\tau\>\nonumber\\
&&\quad-\frac1{\lam_1}\<F,X_1\>-\frac{\<X_2,\tau\>}{\lam_1\<X_1,\tau\>}\<F,X_2\>\qfq x\in\Ga_b.\label{L3.21}\eeq

We now integrate (\ref{4.12}) over $\Sigma_2$ and use (\ref{Pi26})
to obtain
\beq2(W,-\L_0W)_{\LL^2(\Sigma_2,T\Sigma_2)}&&\geq\si_{\varepsilon_0}\|W\|^2_{\LL^2(\Sigma_2,T\Sigma_2)}
+\int_{\Ga_b}\frac{e^{-s\lam_2}\Pi(\a_t,\a_t)}{\<X_2,\nu\>|\a_t|^2}\<W,X_1\>^2d\Ga\nonumber\\
&&\quad+\int_{\Ga_{b-\varepsilon_0}}\frac{e^{-s\lam_2}}{\<X_2,\nu\>}[\frac{\Pi(\a_t,\a_t)}{|\a_t|^2}\<W,X_1\>^2-\lam_2\<W,\nu\>^2]d\Ga.\quad\quad\label{4.13}\eeq
It follows from (\ref{4.13}), (\ref{L3.17}), and (\ref{L3.21}) that
\beq&&\|Dw\|^2_{\LL^2(\Sigma_2,T\Sigma_2)}+\int_{\Ga_b}\<Dw,\tau\>^2d\Ga\leq C(\|w\|^2_{\LL^2(\Sigma_2)}+\|f\|^2_{\LL^2(\Sigma_2)}+\|F\|^2_{\WW^{1,2}(\Sigma_2,T\Sigma_2)}\nonumber\\
&&\quad+\|F\|^2_{\LL^2(\Ga_b,TM)}+\|F\|^2_{\LL^2(\Ga_{b-\varepsilon_0},TM)}+\|Dw\|^2_{\LL^2(\Ga_{b-\varepsilon_0},TM)})\nonumber\\
&&\leq C(\|w\|^2_{\LL^2(\Sigma_2)}+\|f\|^2_{\LL^2(\Sigma_2)}+\|F\|^2_{\WW^{1,2}(\Sigma_2,T\Sigma_2)}+\|Dw\|^2_{\LL^2(\Ga_{b-\varepsilon_0},TM)}).\label{4.18}\eeq
Thus (\ref{4.9}) follows from (\ref{4.18}) and the boundary data in (\ref{4.8}).

{\bf Step 2.}\,\,\,We prove (\ref{w4.9}) in the case of $m=1.$ The case of $m\geq2$ can be treated by the similar arguments.

Let $W$ be given in (\ref{M4.12}). Suppose that $\Phi_2$ and $\Phi_2$ are given in (\ref{e2.37}) such that
$$\Pi(\Phi_i,\Phi_i)>0\qfq x\in\Sigma_2$$ and $\Phi_2,$ $\Phi_2$ forms a basis of vector fields on $\Sigma_2.$ By a similar computation as in (\ref{L2.44}), we have
\beq\L_{Z_i}[\Phi_i,W]&&=e^{-s\lam_2}[(\Phi_i f_2-h_{i1}f_2-h_{i2}f_2+\<H_i,W\>)X_2\nonumber\\
&&\quad+\lam_2(\Phi_i f_2-\<D\div_g\Phi_i,W\>)X_2],\label{L3.24}\eeq where $Z_i,$ $h_{2i},$ and $H_i$ are the same as in (\ref{L2.44}), and
$$f_2=-\div_gQF,\quad f_2=-(\tr_g\Pi)w+f\qfq x\in\Sigma_2.$$ By a similar argument as in Step 1 yields
\beq&&2([\Phi_i,W],-\L_{Z_i}[\Phi_i,W])_{\LL^2(\Sigma_2,T\Sigma_2)}\nonumber\\
&&\geq\si_{\varepsilon_0}\|[\Phi_i,W]\|^2_{\LL^2(\Sigma_2,T\Sigma_2)}
+\int_{\Ga_b\cup\Ga_{b-\varepsilon_0}}e^{-s\lam_2}P_i(W_\varepsilon)d\Ga,\label{L3.25}\eeq where
\be P_i(W)=\frac1{\<X_2,\nu\>}[\Pi(\tau,\tau)\<[\Phi_i,W],X_1\>^2-\lam_2\<[\Phi_i,W],\nu\>^2],\quad i=1,\,\,2.\label{z2.50}\ee

Since
$$\<X_2,\nu\>>0\qfq x\in\Ga_b,\quad\<X_2,\nu\><0\qfq x\in\Ga_{b-\varepsilon_0},$$ from (\ref{L3.24})-(\ref{z2.50}), we obtain
\beq\|[\Phi_i, W]\|^2_{\LL^2(\Sigma_2,T\Sigma_2)}&&+\int_{\Ga_b}\<D_{\Phi_i}W,X_1\>^2d\Ga\leq C(\|f\|^2_{\WW^{1,2}(\Sigma_2)}+\|F\|^2_{\WW^{2,2}(\Sigma_2,T\Sigma_2)}\nonumber\\
&&\quad+\|w\|^2_{\WW^{1,2}(\Sigma_2)}+\|w\|^2_{\LL^2(\Sigma_2)}+\|Dw\|^2_{\LL^2(\Sigma_2,T\Sigma_2)}),\label{phi3.30}\eeq
for $i=1,$ $2,$ which imply that
$$W\in\WW^{1,2}(\Sigma_2,T\Sigma_2).$$ Thus, by the trace theorem,
\be W\in\LL^2(\Ga_b,TM).\label{W3.28}\ee

It is easy to check that  $\tau=\a_t/|\a_t|,$ $X_2$ forms a vector basis along $\Ga_b$ by (\ref{tau3.22}). From (\ref{Ga3.21}), we have
$$Dw=p_1\tau+p_2X_2\qfq x\in\Ga_b,$$ where
$$p_1=\frac1{1-\<\tau,X_2\>^2}[\tau(w)-\<\tau,X_2\>\<F,X_2\>],\quad p_2=\frac1{1-\<\tau,X_2\>^2}[-\<\tau,X_2\>\tau(w)+\<F,X_2\>].$$
Thus
\beq|\<D_\tau Dw,X_1\>|&&\geq|\<\tau,X_1\>||\tau(p_1)|-C(|Dw|+|F|)\nonumber\\
&&\geq |\<\tau,X_1\>||\tau\tau(w)|-C(|\tau(w)|+|F|+|DF|)\qfq x\in\Ga_b.\label{Ga3.28}\eeq

Next, we compute
\beq\<D_{\Phi_1}(\nabla\n W),X_1\>&&=\Phi_1(\lam_1\<W,X_1\>)-\<\nabla\n W,D_{\Phi_1}X_1\>\nonumber\\
&&=\Phi_1(\lam_1)\<W,X_1\>+\<\nabla\n D_{\Phi_1}W,X_1\>+\lam_1\<W,D_{\Phi_1}X_1\>-\<\nabla\n W,D_{\Phi_1}X_1\>.\nonumber\eeq
Thus
\beq\lam_1\<D_{\Phi_1}W,X_1\>&&=\<\nabla\n D_{\Phi_1}W,X_1\>=\<D_{\Phi_1}(Dw-F),X_1\>\nonumber\\
&&\quad-\Phi_1(\lam_1)\<W,X_1\>-\lam_1\<W,D_{\Phi_1}X_1\>+\<\nabla\n W,D_{\Phi_1}X_1\>\nonumber\eeq which yields, by (\ref{W3.28}) and (\ref{Ga3.28}),
\beq|\<D_{\Phi_1}W,X_1\>|&&\geq|\a_t||\<D_\tau Dw,X_1\>|-C(|W|+|DF|)\nonumber\\
&&\geq\si|\tau\tau(w)|-C(|W|+|\tau(w)|+|F|+|DF|\qfq x\in\Ga_b,\nonumber\eeq since $\Phi_1=|\a_t|\tau.$
Thus (\ref{w4.9}) follows from (\ref{phi3.30}) in the case of $m=1.$
\end{proof}

{\bf Proof of Theorem \ref{t4.1}}\,\,\, We solve problem (\ref{4.7}) on the region $\Sigma_1$ to have a solution $w_2$ and then solve problem (\ref{4.8}) on $\Sigma_2$ with the data
 $$ q_0=w_2,\quad q_2=\<Dw_2,Q\nabla\n \a_t\>\qfq x\in\Ga_{b-\varepsilon_0},$$ to obtain the solution $w_2.$
Furthermore, paste two solutions together to
have a solution to problem (\ref{4.1}) on the region $S.$ Clearly, the solution meets our needs.\hfill$\Box$\\

{\bf Proof of Theorem \ref{t4.2}}\,\,\,From Theorem \ref{t4.1} and \cite[Theorem 4.3]{Yao2017}, we have
\beq c\Ga_2(w,S)&&\leq\|w\|^2_{\WW^{2,2}(\Sigma_1)}+\|f\|^2_{\WW^{1,2}(\Sigma_1)}+\|F\|^2_{\WW^{2,2}(\Sigma_1,T\Sigma_1)}\nonumber\\
&&\leq\|w\|^2_{\WW^{2,2}(S)}+\|f\|^2_{\WW^{1,2}(S)}+\|F\|^2_{\WW^{2,2}(S,TS)}\nonumber\\
&&\leq C(\Ga_2(w,S)+\|f\|^2_{\WW^{1,2}(S)}+\|F\|^2_{\WW^{2,2}(S,TS)}).\nonumber\eeq
\hfill$\Box$

\setcounter{equation}{0}
\section{Proofs of Theorems 1.1-1.3 in Section 1}
\def\theequation{6.\arabic{equation}}
\hskip\parindent
Let
$$U\in T^2_{\sym}S.$$ Consider problem
\be\left\{\begin{array}{l}Dv=\nabla\n V+F\qfq x\in S,\\
\div_gV+v\tr_g\Pi=f\qfq x\in S,\end{array}\right.\label{6.1}\ee
where $(V,v)\in\WW^{1,2}(S,T)\times\WW^{1,2}(S)$ is the unknown and
\be F=Q[D(\tr_g U)-\div_gU],\quad f=-\tr_gU(Q\nabla\n\cdot,\cdot)\qfq x\in S.\label{6.7}\ee

For $y\in\WW^{1,2}(S,\R^3),$ let
\be 2v=\nabla y(e_2,e_2)-\nabla y(e_2,e_2)\qfq x\in S,\label{6.2}\ee
\be V=(\nabla\n)^{-1}(Dv-F)\qfq x\in S,\label{6.3}\ee where $e_2,$ $e_2$ is an orthonormal basis of $T_xS$ with positive orientation.

 By\cite[Section 2]{Yao2017}, there is a $y\in\WW^{1,2}(S,\R^3)$ to solve problem (\ref{01}) if and only if $(v,V),$ being given in (\ref{6.2}) and (\ref{6.3}), solves problem (\ref{6.1}). In that case, we have
\be\left\{\begin{array}{l}\nabla_{e_1}y=U(e_1,e_1)e_1+[v+U(e_1,e_2)]e_2-\<QV,e_1\>\n,\\
\nabla_{e_2}y=[-v+U(e_1,e_2)]e_1+U(e_2,e_2)e_2-\<QV,e_2\>\n,\end{array}\right.\qfq x\in S.\label{y6.6}\ee
 Moreover, from \cite[Theorem 2.1]{Yao2017},   $(V,v)$ is a solution to problem (\ref{6.1}) if and only if $v$ solves problem
 \be\<D^2v,Q^*\Pi\>+\frac1\kappa X_0v+v\kappa\tr_g\Pi=\kappa f+\frac1\kappa\<X_0,F\>+\<DF,Q^*\Pi\>\qfq x\in S,\label{6.6}\ee where $f$ and $F$ are given in (\ref{6.7}).\\

{\bf Proof of Theorem \ref{t1.1}}\,\,\, We solve the degenerated hyperbolic problem
\be\left\{\begin{array}{l}\<D^2v,Q^*\Pi\>+\dfrac1\kappa X_0v+\kappa(\tr_g\Pi)v=\kappa f+\dfrac1\kappa\<X_0,F\>+\<DF,Q^*\Pi\>\,\,x\in S,\\
v=\<Dv,Q\nabla\n\a_t\>=0\qfq x\in\Ga_0,\end{array}\right.\nonumber\ee by Theorem \ref{t4.1}, to have the solution $v\in\WW^{m+1,2}(S).$ Then, set
$$V=(\nabla\n)^{-1}(Dw-F)\qfq x\in S.$$
It follows from  \cite[Theorem 2.1]{Yao2017} that problem (\ref{01}) admits a solution $y\in\WW^{m,2}(S,\R^3)$ that satisfies (\ref{y6.6}) on the region $S.$

Let $y=W+w\n,$ where $w=\<y,\n\>\in\WW^{m,2}(S).$ Since
$$\sym DW=U-w\Pi\qfq x\in S,$$ it follows from \cite[Lemma 4.3]{Yaobook} that
$$W\in\WW^{m+1,2}(S,TS).$$
\hfill$\Box$

{\bf Proof of Theorem 1.2}\,\,\,Let $y\in\WW^{2,2}(S,\R^3)$ be an  infinitesimal isometry. Let
$$ 2v=\nabla y(e_2,e_2)-\nabla y(e_2,e_2),\quad V=(\nabla\n)^{-1}Dv\qfq x\in S,$$ where $e_2,$ $e_2$ is an orthonormal basis of $T_xS$ with positive orientation. It follows from (\ref{y6.6}) and (\ref{6.6}) that
\be\left\{\begin{array}{l}\nabla_{e_2}y=ve_2-\<QV,e_2\>\n,\\
\nabla_{e_2}y=-ve_2-\<QV,e_2\>\n\end{array}\right.\qfq x\in S.\label{6.12}\ee and
 \be\<D^2v,Q^*\Pi\>+\frac1\kappa X_0v+v\kappa\tr_g\Pi=0\qfq x\in S.\label{6.13}\ee
From (\ref{6.12}), (\ref{6.13}),  Theorems \ref{t4.1}, and \ref{t4.2}, we have
\be \|y\|^2_{\WW^{2,2}(S)}\leq C(\|v\|^2_{\WW^{1,2}(S)}+\|V\|^2_{\WW^{1,2}(S,TS)})\leq C\Ga_2(v,S)\qfq y\in\V(S,\R^3),\label{6.14}\ee where $\Ga_2(v,S)$ is given in (\ref{n3.7}) with $w$ replaced by $v.$

For given $\varepsilon>0,$ we take $q_{0\varepsilon}\in\WW^{m+2,2}(\Ga_0)$ and $q_{1\varepsilon}\in\WW^{m+1,2}(\Ga_0)$ such that
\be\|v-q_{0\varepsilon}\|^2_{\WW^{2,2}(\Ga_0)}+\|\<Dv,Q\nabla\n\a_t\>-q_{1\varepsilon}\|^2_{\WW^{1,2}(\Ga_0)}\leq\varepsilon.\label{6.15}\ee
Then we solve problem
\be\left\{\begin{array}{l}\<D^2v,Q^*\Pi\>+\dfrac1\kappa X_0v+\kappa(\tr_g\Pi)v=0\qfq x\in S,\\
v=q_{0\varepsilon},\quad\<Dw,Q\nabla\n\a_t\>=q_{1\varepsilon}\quad x\in\Ga_0\end{array}\right.\nonumber\ee to obtain the solution $v_\varepsilon.$
Thus there is  $y_\varepsilon\in\V(S,\R^3)$ with $2v_\varepsilon=\nabla y_\varepsilon(e_2,e_2)-\nabla y_\varepsilon(e_2,e_2).$ By Theorem \ref{t4.1}, $y_\varepsilon\in\WW^{m+2,2}(S,\R^3),$ and, thus, by the imbedding theorem (see \cite[P. 158]{GNT})
$$y_\varepsilon\in\V(S,\R^3)\cup\CC^m_B(S,\R^3).$$ Finally, it follows from Theorem \ref{t4.2}, (\ref{6.14}), and (\ref{6.15}) that
$$ \|y-y_\varepsilon\|^2_{\WW^{2,2}(S)}\leq C\varepsilon.$$ The proof is complete.\hfill$\Box$\\

{\bf Proof of Theorem \ref{t1.3}}\,\,\,As in $\cite{HoLePa}$ we conduct in $2\leq i\leq m.$  Let
$$y_\varepsilon=\sum_{j=0}^{i-1}\varepsilon^jz_j$$ be an $(i-1)$th order isometry of class $\CC^{2+4(m-i+1)}_B(S,\R^3),$ where $z_0=\id$ and $z_2=y$ for some $i\geq2.$ Then
$$\sum_{j=0}^k\na^Tz_j\na z_{k-j}=0\qfq 1\leq k\leq i-1.$$

Next, we shall find out $z_i\in\CC^{2+4(m-i)}_B(S,\R^3)$ such that
$$\phi_\varepsilon=y_\varepsilon+\varepsilon^iz_i$$ is an $i$th order isometry.
By Corollary  $\ref{c1.1}$ there exists  a solution $z_i\in\CC^{2+4(m-i)}_B(S,\R^3)$ to problem
$$\sym \nabla z_i=-\frac{1}{2}\sym\sum_{j=1}^{i-1}\na^Tz_j\na z_{i-j}$$ which satisfies
\beq\|z_i\|_{\CC^{2+4(m-i)}_B(S,\R^3)}&&\leq C\|\sum_{j=1}^{i-1}\sym\na^Tz_j\na z_{i-j}\|_{\CC^{2+4(m-i)+3}_B(S,\R^3)}\nonumber\\
&&\leq C\sum_{j=1}^{i-1}\|z_j\|_{\CC^{2+4(m-i+1)}_B(S,\R^3)}\|z_{i-j}\|_{\CC^{2+4(m-i+1)}_B(S,\R^3)}.\nonumber\eeq
The conduction completes. \hfill$\Box$\\

{\bf Compliance with Ethical Standards}

Conflict of Interest: The author declares that there is no conflict of interest.

Ethical approval: This article does not contain any studies with human participants or animals performed by the authors.

 \end{document}